\documentclass{amsart}

\makeatletter
\@namedef{subjclassname@2020}{%
  \textup{2020} Mathematics Subject Classification}
\makeatother % Cambiar el año de la MSC, de 2010 a 2020

\usepackage{amsthm,amssymb,amsfonts,latexsym,mathtools,thmtools,lscape}
\usepackage[T1]{fontenc}
\usepackage{tikz-cd} % Commutative diagrams.
\usepackage{enumitem} % Customization of items.
\usepackage{hyperref} %hyperlinks
\usepackage{hyperref}
\hypersetup{
    colorlinks=true,
    linkcolor=blue,
    filecolor=blue,      
    urlcolor=blue,
    linktocpage=true
}

\newtheorem{theorem}{Theorem}[section]
\newtheorem{lemma}[theorem]{Lemma}

\theoremstyle{definition}
\newtheorem{definition}[theorem]{Definition}
\newtheorem{proposition}[theorem]{Proposition}
\newtheorem{example}[theorem]{Example}
\newtheorem{remark}[theorem]{Remark}
\newtheorem{corollary}[theorem]{Corollary}

\numberwithin{equation}{section}

\begin{document}

\title[Burchnall-Chaundy theory for quadratic algebras]{A first approach to the Burchnall-Chaundy theory for quadratic algebras having PBW bases}

\author{Arturo Ni\~no}
\address{Universidad Nacional de Colombia - Sede Bogot\'a}
\curraddr{Campus Universitario}
\email{dieaninotor@unal.edu.co}

\author{Mar\'ia Camila Ram\'irez}
\address{Universidad Nacional de Colombia - Sede Bogot\'a}
\curraddr{Campus Universitario}
\email{macramirezcu@unal.edu.co }

\author{Armando Reyes}
\address{Universidad Nacional de Colombia - Sede Bogot\'a}
\curraddr{Campus Universitario}
\email{mareyesv@unal.edu.co}

\thanks{The authors were supported by the research fund of Faculty of Science, Code HERMES 53880, Universidad Nacional de Colombia - Sede Bogot\'a, Colombia.}

\subjclass[2020]{11E88, 16S36, 16S37, 16W50, 17A45}

\keywords{Burchnall-Chaundy theory, Sylvester matrix, resultant, quadratic algebra, PBW basis}

\date{}

\dedicatory{Dedicated to the memory of Emma Previato}

\begin{abstract}

In this paper, we present a first approach toward a Burchnall-Chaundy theory for the skew Ore polynomials of higher order generated by quadratic relations defined by Golovashkin and Maksimov \cite{GolovashkinMaksimov1998}.  

\end{abstract}

\maketitle

%\tableofcontents

\section{Introduction}\label{introduction}

Burchnall and Chaundy written a series of papers \cite{BC1923, BC1928, BC1931} where they were interested in to find an algebraic curve that vanishes on two commuting differential operators. The ring of differential operators considered by them is given by the skew polynomial ring or Ore extension (introduced by Ore \cite{Ore1931, Ore1933}) of derivation type  $C^{\infty}(\mathbb{R}, \mathbb{C})[D;{\rm id}, \delta]$, where $\delta$ is the ordinary derivation. They found that given $P=\sum_{i=0}^{n}p_{i}D^{i}$ and $Q = \sum_{i=0}^{m}q_{i}D^{i}$ two differential operators such that $PQ-QP=0$, where $p_{i}$, $q_{i}$ are complex valued functions, there exists an annihilating complex algebraic curve $\mathcal{BC}$ determinated by some polynomial $F(x,y)$ such that $F(P,Q)=0$. The points on this curve have coordinates which are exactly the eigenvalues associated with the operators $P$ and $Q$. Since then, in the literature this theory is known as \textit{Burchnall-Chaundy theory} ($\mathcal{BC}$ theory), and the curve $F(x,y)$ is called {\em Burchnall-Chaundy curve}.  

The original proof of Burchnall and Chaundy relies in analytical methods. However, several authors have formulated different algebraic approaches toward this theory for families of noncommutative rings related with Ore extensions (e.g., Amitsur  \cite{Amitsur1958}, Flanders \cite{Flanders1955}, Carlson and Goodearl \cite{CarlsonGoodearl1980, Goodearl1983}, Hellstrom and Silvestrov \cite{Hellstrom2, HellstromSilvestrov2007}, Bell and Small \cite{Bell2009, BellSmall2004}, Richter \cite{RichterPseudo2016}, Ni\~no and Reyes \cite{NinoReyes2023}, Reyes and Su\'arez \cite{ReyesSuarez2018}, and references therein). For instance, with the aim of constructing the Burchnall-Chaundy curve in the setting of some families Ore extensions, the notions of \textit{determinant polynomial, subresultant} and {\em resultant} have been investigated by Li in his PhD thesis \cite{LiPhD1996} (see also \cite{Li1998}), Richter in his PhD thesis \cite{RichterPhD2014}, Richter and Silvestrov \cite{RichterSilvestrov2012}, and Larsson \cite{Larsson2014}. It is important to note that the concepts of resultant and determinant polynomials are intimately related since both aim to encode the common factors of pairs (right factors in the noncommutative case) of polynomials. For instance, the theory of the commutative subresultant has been of interest since the work of Collins, Brown and Traub \cite{BrownTraub1971, Collins1967, Loos1982, Mishra1993}, which seeks to establish algorithms that solve the problem of finding greatest common divisors in rings of polynomial type, while the concept of multivariate resultant serves as a criterion for determining the existence of common factors in a pair of homogeneous polynomials \cite{Cayley1848, Coxetal2015}. For example, Chardin \cite{Chardin1991} presented a subresultant theory for linear ordinary differential operators, while Carra'-Ferro \cite{Carra-Ferro1994, Carra-Ferro1997}, McCallum \cite{McCallum}, Rueda and Sendra \cite{RuedaSendra}, and Zhang, Yuan and Gao \cite{ZhangYuanGao}, established different results on differential resultants. Several details, from classical ones to those obtained relatively recently, about resultant theory and their importance in physical research were presented by Morozov and Shakirov \cite{Morozov2010}. 

Motivated by the algebraic approach to the Burchnall-Chaundy theory for noncommutative algebras such as described above, our aim in this paper is to present a first approach to the $\mathcal{BC}$ theory for {\em skew Ore polynomials of higher order generated by homogenous quadratic relations} defined and studied by Golovashkin and Maksimov in several papers \cite{GolovashkinMaksimov1998, GolovashkinMaksimov2005, Maksimov2000, Maksimov2002}. Since some of its ring-theoretical, homological, geometrical and combinatorial properties have been investigated (e.g., \cite{ChaconPhD2022, ChaconReyes2023b, Fajardoetal2020, NinoRamirezReyes2020, NinoReyes2020, NinoReyes2023, Rosengren2000} and references therein), this article can be considered as a contribution to the study of algebraic characterizations of these objects.

The organization of the paper is as follows. Section \ref{preliminaries} contains some preliminaries about quadratic algebras that will be studied throughout the paper. In Section \ref{productsofelements}, Propositions \ref{proposition1NR}, \ref{proposition2NR}, \ref{proposition3NR}, \ref{proposition4NR}, and \ref{proposition5NR} present combinatorial properties on products of elements in these algebras. Next, in Section \ref{SMR} we consider the notions of Sylvester matrix and resultant for quadratic algebras with the aim of determining whether this object characterize common right factors of polynomials. In Section \ref{DeterminantPolynomials}, by using the concept of determinant polynomial, we formulate the version of $\mathcal{BC}$ theory for these algebras (Theorem \ref{BC_corollary}). Finally, Section \ref{ExamplesBC} contains illustrative examples of the results formulated in the previous sections.

Throughout the paper, the term ring means an associative ring (not necessarily commutative) with identity. The symbol $\Bbbk$ denotes an arbitrary field, and all algebras are $\Bbbk$-algebras. The letters $\mathbb{N}$ and $\mathbb{Z}$ denote the set of natural numbers including zero, and the ring of integer numbers, respectively. $M_{r\times c}(R)$ means the ring of matrices of size $r\times c$ ($r\leq c$) with entries in the ring $R$.

\section{Preliminaries}\label{preliminaries}

For a ring $R$, the {\em Ore extensions} introduced by Ore \cite{Ore1931, Ore1933} consist of the uniquely representable elements $r_0 + r_1x + \dotsb + r_kx^k$, $k = k(r) = 0, 1, 2, \dotsc, r_i\in R$, with the commutation relation $xr = \sigma(r)x + \delta(r)$, where $\sigma$ is an endomorphism of $R$ and $\delta$ is a $\sigma$-derivation of $R$. Different generalizations of these objects, called {\em skew Ore polynomials}, have been introduced and studied by several authors (e.g., Cohn \cite{Cohn1961, Cohn1971}, Dumas \cite{Dumas1991}, and Smits \cite{Smits1968}) considering the commutation relation $xr = \Psi_1(r)x + \Psi_2(r)x^2 + \dotsb$, where the $\Psi$'s are endomorphisms of $R$. Nevertheless, there are cases of quadratic algebras such as Clifford algebras, Weyl-Heisenberg algebras, and Sklyanin algebras, in which this commutation relation is not sufficient to define the noncommutative structure of the algebras since a free non-zero term $\Psi_0$ is required (e.g., Ostrovskii and Samoilenko \cite{OstrovskiiSamoilenko1992}). Precisely, {\em skew Ore polynomials of higher order} with commutation relation with this free term, that is, $xr = \Psi_0(r) + \Psi_1(r)x + \dotsb + \Psi_n(r)x^n + \dotsb$, were studied by Maksimov \cite{Maksimov2000}, where, for every $r, s \in R$, the free term $\Psi_0$ satisfies the relation
\[
\Psi_0(rs) = \Psi_0(r)s + \Psi_1(r)\Psi_0(s) + \Psi_2(r)\Psi_0^{2}(s) + \dotsb 
\]

or the equivalent operator equation $\Psi_0 r = \Psi_0(r) + \Psi_1(r)\Psi_0 + \Psi_2(r)\Psi_0^{2}$, where $r$ is considered as the operator of left multiplication by $r$ on $R$. One may consider $\Psi_0$ as a singular differentiation operator with respect to $\Psi_1, \Psi_2, \dotsc$, but where $\Psi_1$ need not be an endomorphism of $R$. Notice that a general algebra of skew polynomials with indeterminate $x$ over $R$ is also determined by certain linear operators $\Psi_k:R \to R$ such that for each $r\in R$, there is a unique representation $xr = \Psi_0(r) + \Psi_1(r)x + \dotsb + \Psi_{m(r)}x^{m(r)}$. Of course, if $R$ is a ring with one generator, then the algebra has a PBW basis. If $m_0 = 1$, then we obtain the classical Ore extensions \cite{Ore1933}. It it important to say that the associativity and uniqueness of the representation of the product $xr$ guarantee conditions on the set of operators $\{\Psi_k\}_{k\in \mathbb{N}}$ which are necessary and sufficient to determinate the algebra of skew polynomials (for more details, see \cite{Maksimov2000, Maksimov2002}).

Golovashkin and Maksimov \cite{GolovashkinMaksimov2005} investigated the representation of the $\Bbbk$-algebra $Q(a, b, c)$ (${\rm char}(\Bbbk) = 0$) defined by a {\em quadratic relation} in two generators $x, y$ given by 
\begin{equation}\label{GolovashkinMaksimov2005(1)}
        yx = ax^2 + bxy + cy^2,
\end{equation}
    
as an algebra of Ore polynomials of higher degree with commutation relation (\ref{GolovashkinMaksimov2005(1)}) with $a, b, c$ belong to $\Bbbk$. As one can check, the algebra generated by the relation is represented in the form of an algebra of skew Ore polynomials of higher order if the elements $\{x^my^n\}$ form a linear basis of the algebra. Hence, this algebra can be defined by a system of linear mappings $\Psi_0, \Psi_1, \Psi_2, \dotsc$ of the algebra of polynomials $\Bbbk[x]$ into itself such that for an arbitrary element $p(x)\in \Bbbk[x]$, $yp(x) = \Psi_0(p(x)) + \Psi_1(p(x))y + \dotsb + \Psi_k(p(x))y^{k}$, $k = k(p(x)),\ k = 0, 1, 2\dotsc$ If this representation exists, then one can to obtain the relations between the operators $\Psi_0, \Psi_1, \Psi_2, \dotsc$ They found conditions for such an algebra $Q(a, b, c)$ to be expressed as a skew polynomial with generator $y$ over the polynomial ring $\Bbbk[x]$ (cf. \cite{GolovashkinMaksimov1998}), and proved that these conditions are equivalent to the existence of a {\em Poincar\'e-Birkhoff-Witt} (PBW for short) basis, i.e., basis of the form $\{x^my^n\}$. A useful tool in the study of PBW bases for quadratic algebras defined by (\ref{GolovashkinMaksimov2005(1)}) is the matrix given by
\begin{equation}\label{matrixPBW}
    M := \begin{pmatrix}
\;\;\,b & a \\ -c & 1
\end{pmatrix},
\end{equation}

which is called the {\em companion} matrix for (\ref{GolovashkinMaksimov2005(1)}) \cite[Section 2.2]{GolovashkinMaksimov2005}. If the lower-right elements of the matrices $M^l$ does not vanish for all $l\in \mathbb{N}$, then $Q(a, b, c)$ has a PBW basis of the form $\{x^my^n\mid m, n \in \mathbb{N}\}$ \cite[Proposition 4]{GolovashkinMaksimov2005}. 

If $b+ac\neq 0$ a necessary and sufficient condition for the monomials $\{x^my^n\mid m, n \in \mathbb{N}\}$ form a basis of $Q(a, b, c)$ is precisely that the lower-right elements of the matrices $M^l$ does not vanish for all $l\in \mathbb{N}$. If $a = c = 1$ and $b = -1$, then the elements $\{x^my^n\mid m, n \in \mathbb{N}\}$ are linearly independent but do not form a PBW basis of $Q(a, b, c)$ \cite[Propositions 5 and 10]{GolovashkinMaksimov2005}. Notice that when two of the three coefficients $a, b$, and $c$ are equal to zero, the resulting algebras are given by three types: (i) $yx = ax^2$ (ii) $yx = cy^2$ and (iii) $yx = bxy$. The set $\{x^my^n\mid m, n\in \mathbb{N}\}$ is a PBW basis for algebras (i) and (ii), while the algebra (iii) is an Ore extension \cite[Section 1.5]{GolovashkinMaksimov2005}. 

From the facts above, it follows that every polynomial $f(x, y) \in Q(a, b, c)$ has a unique representation as a finite sum of the monomials of the PBW basis in the form $f(x,y) = \sum_{i,j} a_{i,j}x^{i}y^{j},\quad a_{i, j}\in \Bbbk$. As usual, the {\em degree} of every monomial $x^{i}y^{j}$ is defined to be $i+j$, and the {\em degree} of $f(x,y)$ is the greatest degree of the monomials that appear in the expansion of $f(x,y)$. This degree is known as the \textit{total degree} of $f(x,y)$. As in the commutative polynomial ring $\Bbbk[x, y]$, we also have the expression $f(x,y)=\sum_{i=0}^{n} p_{i}(x)y^{i}$, where the $p_i(x)$'s are elements of the commutative polynomial ring $\Bbbk[x]$. We call this representation of $f(x,y)$ the \textit{the normal form} of $f(x,y)$, and we will refer to $n$ as the degree of $f$ with respect to the indeterminate $y$.

\subsection{Products of homogeneous elements}\label{productsofelements}

We consider some combinatorial properties of products of elements for several values of $a, b$ and $c$ in the defining relation (\ref{GolovashkinMaksimov2005(1)}). It is possible that these are found in the literature; however, we could not find them explicitly somewhere, so we provide their corresponding proofs.

Let $Q(a,b,c)$ be a quadratic algebra having a PBW basis of the form $\{x^{m}y^{n}\mid m, n\in \mathbb{N}\}$. We divide our treatment in the Cases \ref{a-1c}, \ref{a00} and \ref{0b0}.

\subsubsection{Case $Q(a,-1,c)$}\label{a-1c}

As we mentioned before, Golovashkin and Maksimov proved that if $b=-1$ and $ac=1$, then the set $\{x^{n}y^{m}\mid n,m\in \mathbb{N}\}$ is not a PBW basis, so we consider the case $ac\neq 1$. From \cite[Proposition 3.8]{ChaconReyes2023}, we know that the set $\{x^{n}y^{m}\mid n,m\in \mathbb{N}\}$ is a PBW basis for $Q(a,-1,c)$, and by \cite[Lemma 3.9]{ChaconReyes2023}, the following commuting rules hold in this algebra:
\begin{enumerate}
    \item [\rm (1)] If $k$ is even, then $yx^{k}=x^{k}y$ and $y^{k}x=xy^{k}$.
    \item [\rm (2)] If $k$ is odd, then $yx^{k}=ax^{k+1}-x^{k}y+cx^{k-1}y^{2}$ and $y^{k}x=ax^{2}y^{k-1}-xy^{k}+cy^{k+1}$.
\end{enumerate}

We can go further and prove the facts established in Propositions \ref{proposition1NR}, \ref{proposition2NR} and \ref{proposition3NR}.

\begin{proposition}\label{proposition1NR}
In $Q(a,-1,c)$, the following formulas hold:
\begin{enumerate}
    \item [\rm (1)] $y^{n}x^{m}=x^{m}y^{n}$, if $m$ is even or $n$ is even.
    \item [\rm (2)] $y^{n}x^{m}=ax^{m+1}y^{n-1}-x^{m}y^{n}+cx^{m-1}y^{n+1}$, if $m$ and $n$ are odd.
\end{enumerate}
\end{proposition}
\begin{proof}
\begin{enumerate}
    \item [\rm (1)] If $n$ is even, then $y^{n}$ commutes with $x$, whence $y^{n}$ belongs to the center of $Q(a,-1,c)$. In particular, $y^{n}x^{m}=x^{m}y^{n}$. The argument for $x^{m}$ is analogous.
    \item [\rm (2)] In this case, we have
    \begin{align*}
        y^{n}x^{m}&=y^{n}x^{m-1}x = x^{m-1}y^{n}x\\
        &=x^{m-1}(ax^{2}y^{n-1}-xy^{n}+cy^{n+1}) = ax^{m+1}y^{n-1}-x^{m}y^{n}+cx^{m-1}y^{n+1}.
    \end{align*}
\end{enumerate}
\end{proof}

In the following, the symbol $\lfloor \square \rfloor$  denote the greatest integer less than or equal to $\square$.

\begin{proposition}\label{proposition2NR}
Suppose that $n$ is an odd number. 
\begin{itemize}
    \item[\rm (1)] If $m$ is an even number, then
    \begin{align*}
        y^{n}\left(\sum_{j=0}^{m}e_{j}x^{m-j}y^{j} \right) = &\ \sum_{j=0}^{m}(-1)^{j}e_{j}x^{m-j}y^{n+j}\\
        & +\sum_{j=0}^{\frac{m}{2}-1}e_{(2j+1)}\left(ax^{m-2j}y^{n+2j}+cx^{m-(2j+2)}y^{n+(2j+2)}\right).    
    \end{align*}
    \item [\rm (2)] If $m$ is an odd number, then we have 
    \begin{align*}
        y^{n}\left(\sum_{j=0}^{m}e_{j}x^{m-j}y^{j} \right) =  &\ \sum_{j=0}^{\lfloor \frac{m}{2} \rfloor} e_{2j}(ax^{m-2j+1}y^{2j+n-1}\\
        &- x^{m-2j}y^{2j+n}+ cx^{m-(2j+1)}y^{2j+n+1})\\
        &+ \sum_{j=0}^{\lfloor \frac{m}{2} \rfloor}e_{(2j+1)}x^{m-(2j+1)}y^{n+(2j+1)}.
    \end{align*} 
    \end{itemize}
\end{proposition}
\begin{proof}
\begin{enumerate}
    \item [\rm (1)] Suppose that $m$ is an even number. We get
    {\small{
\begin{align*}   
y^{n}\left(\sum_{j=0}^{m}e_{j}x^{m-j}y^{j} \right)&=y^{n}\left(\sum_{j=0}^{\frac{m}{2}}e_{2j}x^{m-2j}y^{2j} \right)+y^{n}\left(\sum_{j=0}^{\frac{m}{2}-1}e_{(2j+1)}x^{m-(2j+1)}y^{2j+1} \right) \\
    &=\sum_{j=0}^{\frac{m}{2}}e_{2j}x^{m-2j}y^{n+2j} +\sum_{j=0}^{\frac{m}{2}-1}e_{(2j+1)}y^{n}x^{m-(2j+1)}y^{2j+1}\\
    &=\sum_{j=0}^{\frac{m}{2}}e_{2j}x^{m-2j}y^{n+2j} \\
    &\ \ +\sum_{j=0}^{\frac{m}{2}-1}e_{(2j+1)}\left(ax^{m-2j}y^{n-1}-x^{m-(2j+1)}y^{n}+cx^{m-(2j+2)}y^{n+1}\right)y^{2j+1}.
\end{align*}
}}

Equivalently, 
{\small{
\begin{align*} 
    y^{n}\left(\sum_{j=0}^{m}e_{j}x^{m-j}y^{j} \right) &=\sum_{j=0}^{\frac{m}{2}}e_{2j}x^{m-2j}y^{n+2j}-\sum_{j=0}^{\frac{m}{2}-1}e_{(2j+1)}x^{m-(2j+1)}y^{n+(2j+1)}\\
    &\ \ +\sum_{j=0}^{\frac{m}{2}-1}e_{(2j+1)} \left(ax^{m-2j}y^{n+2j}+cx^{m-(2j+2)}y^{n+(2j+2)}\right) \\
    &=\sum_{j=0}^{m}(-1)^{j}e_{j}x^{m-j}y^{n+j}\\
    &\ \ +\sum_{j=0}^{\frac{m}{2}-1}e_{(2j+1)}\left(ax^{m-2j}y^{n+2j}+cx^{m-(2j+2)}y^{n+(2j+2)}\right).
\end{align*}
}}

\item [\rm (2)] If $m$ is odd, 
\begin{align*}
    y^{n}\left(\sum_{j=0}^{m}e_{j}x^{m-j}y^{j} \right)&= y^n \left(\sum_{j=0}^{\lfloor \frac{m}{2} \rfloor}e_{2j}x^{m-2j}y^{2j} \right)\\
    & \ \ + y^n \left(\sum_{j=0}^{\lfloor \frac{m}{2} \rfloor}e_{(2j+1)}x^{m-(2j+1)}y^{2j+1} \right),
    \end{align*}

or equivalently,
\begin{align*}
    y^{n}\left(\sum_{j=0}^{m}e_{j}x^{m-j}y^{j} \right) &=  \sum_{j=0}^{\lfloor \frac{m}{2} \rfloor}e_{2j}y^nx^{m-2j}y^{2j} + \sum_{j=0}^{\lfloor \frac{m}{2} \rfloor}e_{(2j+1)}x^{m-(2j+1)}y^{(n+(2j+1))}\\
    &= \sum_{j=0}^{\lfloor \frac{m}{2} \rfloor} e_{2j}(ax^{m-2j+1}y^{2j+n-1} - x^{m-2j}y^{2j+n}\\
    & \ \ + cx^{m-(2j+1)}y^{2j+n+1}) + \sum_{j=0}^{\lfloor \frac{m}{2} \rfloor}e_{(2j+1)}x^{m-(2j+1)}y^{n+(2j+1)}.
\end{align*}
\end{enumerate}
\end{proof}

The following formulas can be used for computing the product of two homogeneous polynomials. 

\begin{proposition}\label{proposition3NR}
Let $f(x,y), g(x,y)\in Q(a, -1, c)$ be two homogeneous polynomials given by  
\begin{equation}\label{fyg}
    f(x,y) = \sum_{i=0}^{m}e_{i}x^{m-i}y^{i},\ \ {\rm and}\ \  g(x,y) = \sum_{j=0}^{n}l_{j}x^{n-j}y^{j}.
\end{equation}

If 
\[
p = \begin{cases}
\lfloor \frac{m}{2} \rfloor, &\text{if $m$ is odd}\\
\frac{m}{2}, &\text{if $m$ is even,}
\end{cases} \quad \text{and} \quad q = \begin{cases}
\lfloor \frac{m}{2} \rfloor, &\text{if $m$ is odd}\\
\frac{m}{2} -1, &\text{if $m$ is even},
\end{cases}
\]

then we have the following equalities:
\begin{enumerate}
\item [\rm (1)] If $n$ is an even number, then
    \begin{align*}
    f(x,y)g(x,y) = &\ \sum^{p}_{i=0} \sum^{\frac{n}{2}}_{j=0} e_{2i}l_{2j}x^{(m+n)-(2i+2j)}y^{2i+2j}\\
    &+ \sum^{q}_{i=0} \sum^{\frac{n}{2}-1}_{j=0} e_{2i+1}l_{2j+1}(ax^{(m+n)-(2i+2j+1)}y^{(2i+2j+1)}\\
    &+ x^{(m+n)-(2i+2j+2)}y^{(2i+2j+2)} + cx^{(m+n) - (2i+2j+3)}y^{(2i+2j+3)}).
\end{align*}

\item [\rm (2)] On the other hand, if $n$ is an odd number, we get
    \begin{align*}
    f(x,y)g(x,y) = &\ \sum^{p}_{i=0} \sum^{\lfloor \frac{n}{2} \rfloor}_{j=0} e_{2i}l_{2j}x^{(m+n)-(2i+2j)}y^{(2i+2j)}\\
    &+ \sum^{q}_{i=0} \sum^{\lfloor \frac{n}{2} \rfloor}_{j=0} e_{2i+1}l_{2j+1}x^{(m+n)-(2i+2j+2)}y^{(2i+2j+2)}.
\end{align*}
\end{enumerate}
\begin{proof}
\begin{enumerate}
    \item [\rm (1)] Suppose that $n$ is even. We have
    \begin{align*}
     f(x,y)g(x,y) = &\ \left(\sum_{i=0}^{m}e_{i}x^{m-i}y^{i}\right)\left(\sum_{j=0}^{n}l_{j}x^{n-j}y^{j}\right)\\
     = &\ \sum_{i=0}^{m}\sum_{j=0}^{n} e_{i}l_{j}x^{m-i}y^{i}x^{n-j}y^{j}\\
    = &\ \sum^{p}_{i=0} \sum^{\frac{n}{2}}_{j=0} e_{2i}l_{2j}x^{m-2i}y^{2i}x^{n-2j}y^{2j}\\
    & + \sum^{q}_{i=0} \sum^{\frac{n}{2}-1}_{j=0} e_{2i+1}l_{2j+1}x^{m-(2i+1)}y^{2i+1}x^{n-(2j+1)}y^{2j+1}\\
    = &\ \sum^{p}_{i=0} \sum^{\frac{n}{2}}_{j=0} e_{2i}l_{2j}x^{(m+n)-(2i+2j)}y^{2i+2j}\\
    & + \sum^{q}_{i=0} \sum^{\frac{n}{2}-1}_{j=0} e_{2i+1}l_{2j+1}x^{m-(2i+1)}(ax^{n-2j}y^{2i} \\
    & -  x^{n-(2j+1)}y^{2i+1} + cx^{n-(2j+2)}y^{2i+2}) y^{2j+1}\\
    = &\ \sum^{p}_{i=0} \sum^{\frac{n}{2}}_{j=0} e_{2i}l_{2j}x^{(m+n)-(2i+2j)}y^{2i+2j}\\
    &+ \sum^{q}_{i=0} \sum^{\frac{n}{2}-1}_{j=0} e_{2i+1}l_{2j+1}(ax^{(m+n)-(2i+2j+1)}y^{(2i+2j+1)}\\
    &+ x^{(m+n)-(2i+2j+2)}y^{(2i+2j+2)} + cx^{(m+n) - (2i+2j+3)}y^{(2i+2j+3)}).
\end{align*}

\item [\rm (2)] Let $n$ be an odd number. We assert that
    \begin{align*}
     f(x,y)g(x,y) = &\ \left(\sum_{i=0}^{m}e_{i}x^{m-i}y^{i}\right)\left(\sum_{j=0}^{n}l_{j}x^{n-j}y^{j}\right) = \sum_{i=0}^{m}\sum_{j=0}^{n} e_{i}l_{j}x^{m-i}y^{i}x^{n-j}y^{j}\\
    = &\ \sum^{p}_{i=0} \sum^{\lfloor \frac{n}{2} \rfloor}_{j=0} e_{2i}l_{2j}x^{m-2i}y^{2i}x^{n-2j}y^{2j}\\
    & + \sum^{q}_{i=0} \sum^{\lfloor \frac{n}{2} \rfloor}_{j=0} e_{2i+1}l_{2j+1}x^{m-(2i+1)}y^{2i+1}x^{n-(2j+1)}y^{2j+1}\\
    = &\ \sum^{p}_{i=0} \sum^{\lfloor \frac{n}{2} \rfloor}_{j=0} e_{2i}l_{2j}x^{(m+n)-(2i+2j)}y^{(2i+2j)}\\
    & + \sum^{q}_{i=0} \sum^{\lfloor \frac{n}{2} \rfloor}_{j=0} e_{2i+1}l_{2j+1}x^{(m+n)-(2i+2j+2)}y^{(2i+2j+2)}.
\end{align*}
\end{enumerate}
\end{proof}
\end{proposition}

\subsubsection{Case $Q(a,0,0)$}\label{a00}

It is clear that this algebra has a PBW basis of the form $\{x^my^n\mid m, n \in \mathbb{N}\}$, and from \cite{ChaconReyes2023}, we know that the relation $y^nx^k = a^nx^{n+k}$ holds for all $n \geq 0$ and $k \geq 1$. 

\begin{proposition}\label{proposition4NR}
If $f(x,y), g(x,y)\in Q(a,0,0)$ are two homogeneous elements given by the expressions 
\begin{equation}\label{fyg}
    f(x,y) = \sum_{i=0}^{m}e_{i}x^{m-i}y^{i}\ \ {\rm and}\ \ 
    g(x,y) = \sum_{j=0}^{n}l_{j}x^{n-j}y^{j},
\end{equation}

then
\begin{align*}
    f(x,y)g(x,y)= \sum_{i=0}^{m}\sum_{j=0}^{n} \left(e_il_ja^i\right) x^{m+n-j} y^j.
\end{align*}
\end{proposition}
\begin{proof}
Since $y^nx^k = a^nx^{n+k}$, for all $n \geq 0$ and $k \geq 1$, it follows that 
\begin{align*}
    f(x,y)g(x,y) = &\ \left(\sum_{i=0}^{m}e_{i}x^{m-i}y^{i}\right)\left(\sum_{j=0}^{n}l_{j}x^{n-j}y^{j}\right) = \sum_{i=0}^{m}\sum_{j=0}^{n} e_il_j x^{m-i}y^ix^{n-j}y^j\\
    = &\ \sum_{i=0}^{m}\sum_{j=0}^{n} e_il_j x^{m-i}\left(a^ix^{n-j+i}\right)y^j\\
    = &\ \sum_{i=0}^{m}\sum_{j=0}^{n} \left(e_il_ja^i\right) x^{m+n-j} y^j.
\end{align*}
\end{proof}

\begin{corollary}
For the element $g(x,y)$ given by the expression (\ref{fyg}), we have 
    \begin{equation*}
        y^{n}g(x,y)= \sum_{j=0} l_ja^n x^{2n-j} y^j.
    \end{equation*}
\end{corollary}
\begin{proof}
The assertion follows from the equalities
\[
y^{n}g(x,y) = y^n \sum_{j=0}^{n}l_{j}x^{n-j}y^{j} = \sum_{j=0}^{n}l_{j}y^n x^{n-j}y^{j} = \sum_{j=0}^{n}l_{j}a^n x^{2n-j}y^{j}.
\]
\end{proof}

\subsubsection{Case $Q(0,b,0)$}\label{0b0}

The algebra $Q(0,b,0)$ is known in the literature as the {\em Manin's plane} or the {\em quantum plane}. Algebraic descriptions of the centralizer of elements belonging to this algebra can be found in Artamanov and Cohn's paper \cite{ArtamanovCohn1999}. It is straightforward to see that this algebra satisfies the relation $y^nx^k = b^{kn}x^ky^n$ holds for all $n,k \geq 0$. Proposition \ref{proposition5NR} extends this relationship by considering the product of two polynomials belonging to this algebra.

\begin{proposition}\label{proposition5NR}
If $f(x,y), g(x,y)\in Q(0,b,0)$ are two homogeneous elements given by
\begin{equation*}
    f(x,y) = \sum_{i=0}^{m}e_{i}x^{m-i}y^{i}\ \ {\rm and}\ \ g(x,y) = \sum_{j=0}^{n}l_{j}x^{n-j}y^{j},
\end{equation*}

then
\begin{align*}
    f(x,y)g(x,y)= \sum_{i=0}^{m}\sum_{j=0}^{n} \left(e_il_j b^{in-ij}\right) x^{(m+n)-(i+j)} y^{i+j}.
\end{align*}
\end{proposition}
\begin{proof}
Since $y^nx^k = b^{kn}x^ky^n$, for all $n,k \geq 0$, it follows that 
\begin{align*}
    f(x,y)g(x,y) &= \left(\sum_{i=0}^{m}e_{i}x^{m-i}y^{i}\right)\left(\sum_{j=0}^{n}l_{j}x^{n-j}y^{j}\right) = \sum_{i=0}^{m}\sum_{j=0}^{n} e_il_j x^{m-i}y^ix^{n-j}y^j\\
    &= \sum_{i=0}^{m}\sum_{j=0}^{n} e_il_j x^{m-i} b^{(i)(n-j)}x^{n-j}y^iy^j\\
    &= \sum_{i=0}^{m}\sum_{j=0}^{n} \left(e_il_j b^{in-ij}\right) x^{(m+n)-(i+j)} y^{i+j}.
\end{align*}
\end{proof}

\begin{corollary}
The following rule of commutation holds for any value of $n$, $m$ and $i$.
    \begin{align*}
        x^{m-i}y^{i}\left(\sum_{j=0}^{n} e_{j}x^{n-j}y^{j} \right)= \sum_{j=0}^{n} e_{j}b^{i(n-j)}x^{m+n-(i+j)}y^{i+j}.
    \end{align*}
\end{corollary}

In the next section we introduce the notions of Sylvester matrix and resultant for quadratic algebras for the algebras considered above.

\section{Sylvester matrix and Resultants}\label{SMR}

In this section, we follow the ideas used in the commutative case to define the Sylvester matrix associated with a pair homogeneous polynomials in several variables \cite{Morozov2010}. We also explore whether this notion determines the existence of common right factors in the context of the algebras of our interest.

\subsection{Sylvester matrix for $Q(a,b,c)$}\label{Sylvestermatrix}

From the properties obtained in Section \ref{productsofelements}, we can see that the product of two homogeneous polynomials is a homogeneous polynomial. In this way, if we define the sets 
\[
H_{n}=\left\{p(x,y)\in Q(a,b,c)\mid p(x,y) \;\,\text{is homogeneous of total degree}\;\, n\right\},
\]

then it is straightforward to see that $H_{n}H_{m}\subseteq H_{n+m}$. Now, since $Q(a,b,c)=\bigoplus\limits_{n\in \mathbb{N}} H_{n}$, it is clear that $\{H_{n}\}_{n\in\mathbb{N}}$ is a graduation for $Q(a,b,c)$.

Related with this, an important fact to formulate a definition of resultant appears when we consider the homomorphism of left $\Bbbk$-modules
\begin{align}\label{homo}
        \phi:H_{n-1}\times H_{m-1}&\longrightarrow H_{m+n-1} \notag \\
        (c(x,y),d(x,y))&\longmapsto c(x,y)f(x,y)+d(x,y)g(x,y),
\end{align}

which is well defined due to the fact that the set $\{x^{i}y^{j}\}_{i,j\geq 0}$ is a PBW basis and the graduation above.

Given two polynomials $f(x, y)$ and $g(x, y)$ in the quadratic algebra $Q(a, b, c)$, the question about the existence of two homogeneous polynomials $c(x,y)$ and $d(x,y)$ such that
\begin{equation}\label{defining_equation}
    c(x, y)f(x, y) + d(x, y)g(x, y)=0,
\end{equation}

can be formulated on the characterization of the kernel of the homomorphism $\phi$ in expression (\ref{homo}). Having in mind that the set $\{x^{n}y^{m}\mid n,m \in \mathbb{N}\}$ is a PBW basis for $Q(a,b,c)$, it follows that $\{x^{i}y^{j}\mid i+j=n\}$ is a basis for the $\Bbbk$-module $H_{n}$. Hence, we can think about defining the matrix that represents the homomorphism $\phi$, and for this we will consider a fixed monomial order on the PBW basis: the lexicographic monomial ordering. In this way, the following is the ordered set that we consider for the PBW basis
{\small{
\begin{align}\label{basis}
    \mathcal{B} = \left\{(x^{n},0), (x^{n-1}y,0),(x^{n-2}y^{2},0),\ldots, (y^{n},0),(0,x^{m}),(0,x^{m-1}y),\ldots,(0,y^{m})\right \}
\end{align}
}}

Now, we define the \textit{Sylvester matrix} for homogeneous elements in quadratic algebras.

\begin{definition}\label{Sylvestermatrix}
    Let $f(x,y)$ and $g(x,y)$ be two homogeneous polynomials in $Q(a,b,c)$ with degree $m$ and $n$, respectively. The \textit{Sylvester matrix} of $f(x,y)$ and $g(x,y)$ is the matrix that represents the homomorphism (\ref{homo}) in the basis given by the set (\ref{basis}). This matrix will be denoted by ${\rm Syl}_{Q(a,b,c)}(f,g)$ and has size $(m+n) \times (m+n)$. The determinant of this matrix will be called \textit{the resultant} of $f(x,y)$ and $g(x,y)$, and it will be denoted by ${\rm Res}_{Q(a,b,c)}$.
\end{definition}

The Sylvester matrix for the algebras of our interest has one of the forms which are presented below. Note that all the formulas deduced in Section \ref{productsofelements} are used to construct these matrices. Let us describe how the Sylvester matrix is constructed: at position $i, j$ ($i$th row and $j$th column), the entry corresponds to the coefficient of the monomial $x^{m+n-i}y^{i-1}$ of the polynomial $x^{n-j}y^{j-1}f$, for the first $n$ columns, and in the last columns, the coefficient of the same monomial but of the polynomial $x^{m-k}y^{k-1}g$. Since quadratic algebras we are considering have PBW basis, the function $\gamma_{i,j}: Q(a,b,c) \to \Bbbk,\ p(x,y) \to c_{i,j}$, where $c_{i,j}$ is the coefficient of the monomial $x^{i}y^{j}$ in the expansion of $p(x,y)$ in terms of the basis $\{x^{i}y^{j}\mid i,j \geq 0\}$, is well defined. With this notation, we can describe the Sylvester matrix ${\rm Syl}_{Q(a,b,c)}(f,g)$ as
{\footnotesize
\begin{align*}
\begin{bmatrix}
        \gamma_{m+n-1,0}(x^{n-1}f)  & \dotsc & \gamma_{m+n-1,0}(y^{n-1}f)& \gamma_{m+n-1,0}(x^{m-1}g) & \dotsc & \gamma_{m+n-1,0}(y^{m-1}g) \\
        \gamma_{m+n-2,1}(x^{n-1}f)  & \dotsc & \gamma_{m+n-2,1}(y^{n-1}f)& \gamma_{m+n-2,1}(x^{m-1}g) & \dotsc & \gamma_{m+n-2,1}(y^{m-1}g)\\
        \gamma_{m+n-3,2}(x^{n-1}f)  & \dotsc & \gamma_{m+n-3,2}(y^{n-1}f)& \gamma_{m+n-3,2}(x^{m-1}g) & \dotsc & \gamma_{m+n-3,2}(y^{m-1}g) \\
        \vdots  & \vdots & & \vdots &  & \vdots\\
        \gamma_{n+1,m-2}(x^{n-1}f) & \dotsc & \gamma_{n+1,m-2}(y^{n-1}f)& \gamma_{n+1,m-2}(x^{m-1}g) & \dotsc & \gamma_{n+1,m-2}(y^{m-1}g)\\
        \gamma_{n,m-1}(x^{n-1}f) & \dotsc & \gamma_{n,m-1}(y^{n-1}f)& \gamma_{n,m-1}(x^{m-1}g) & \dotsc & \gamma_{n,m-1}(y^{m-1}g)\\ 
        \vdots &  \vdots & & \vdots &  & \vdots\\
        \gamma_{0,m+n-1}(x^{n-1}f)  & \dotsc & \gamma_{n,m+n-1}(y^{n-1}f)& \gamma_{n,m+n-1}(x^{m-1}g) & \dotsc & \gamma_{n,m+n-1}(y^{m-1}g)
    \end{bmatrix}.
\end{align*}
}

Let us see some illustrative examples of Definition \ref{Sylvestermatrix}.

\subsubsection{Case $Q(a,-1,c)$}
The Sylvester matrix depends on the parity or oddness of the degrees of the elements $f(x,y) = \sum\limits_{i=0}^{m}e_{i}x^{m-i}y^{i}$ and $g(x,y) = \sum\limits_{j=0}^{n}l_{j}x^{n-j}y^{j}$. The corresponding matrices are shown in (\ref{matrix1of4}), (\ref{matrix2of4}), (\ref{matrix3of4}), and (\ref{matrix4of4}).

\begin{example}
Let $f(x,y)=x^{2}+y^{2}$ and $g(x,y)=xy$ be elements of $Q(a,-1,c)$. Then $c(x,y) = c_{1}x+c_{2}y$ and $d(x,y) = d_{1}x+d_{2}y$, whence $cf = c_{1}x^{3}+c_{2}x^{2}y+c_{1}xy^{2}+c_{2}y^{3}$ and $dg = (d_{1}+d_{2}a)x^{2}y-d_{2}xy^{2}+cd_{2}y^{3}$, and the Sylvester matrix of $f$ and $g$ is given by
\begin{equation*}
    {\rm Syl}_{Q(a,-1,c)}(f,g)=\begin{bmatrix}
    1 & 0 & 0 & \;\;\,0 \\
    0 & 1 & 0 & \;\;\,a \\
    1 & 0 & 0 & -1\\
    0 & 1 & 0 & \;\;\,c
    \end{bmatrix}.
\end{equation*}
\end{example}

\begin{landscape}
\begin{enumerate}
    \item [\rm (1)] If $m$ is even and $n$ is even, then the matrix has the following form:
    {\scriptsize
    \begin{equation}\label{matrix1of4}
    {\rm Syl}_{Q(a,-1,c)}(f,g)=\begin{bmatrix}
    e_{0} & 0 & 0 & \dotsc & 0  & l_{0} & 0 & \dotsc &0\\
    e_{1} & e_{0}+ae_{1} & 0  & \dotsc & 0  & l_{1} & l_{0}+al_{1} & \dotsc & 0\\
    e_{2} & -e_{1} & e_{0}  & \dotsc & 0  & l_{2} & -l_{1} & \dotsc & 0\\
    e_{3} & ce_1+e_2+ae_{3} & e_{1}  & \dotsc & 0  & l_{3} & cl_{1}+l_{2}+al_{3} & \dotsc & 0\\
    e_{4} & -e_3 & e_{4} &\dotsc & 0  & l_{4} & -l_{3} & \dotsc & 0\\
    e_5 & ce_{3}+e_{4}+ae_{5} & e_{5}  & \dotsc & 0 & l_{5} & cl_{3}+l_{4}+al_{5} & \dotsc & 0\\
    \vdots & \vdots & \vdots  & \dotsc & \vdots  & \vdots &\vdots & \vdots & \vdots  \\
    e_{m-2} & -e_{m-3} & e_{m-4}  & \dotsc & 0  & l_{n-3} & cl_{n-5}+l_{n-4}+al_{n-3} & \dotsc &0\\
    e_{m-1} & ce_{m-3}+e_{m-2}+ae_{m-1} & e_{m-3}  & \dotsc & 0 & l_{n-2} & -l_{n-3} & \dotsc & 0\\
    e_{m} & -e_{m-1} & e_{m-2}  & \dotsc & 0 & l_{n-1} & cl_{n-3}+l_{n-2}+al_{n-1} & \dotsc & 0\\
    0 &  ce_{m-1}+e_{m} & e_{m-1}  & \dotsc & e_{0}  & l_{n} & -l_{n-1} & \dotsc & l_{0}\\
    0 & 0 & e_{m}  & \dotsc & e_{1}  & 0& cl_{n-1}+l_{n} & \dotsc & l_{n}\\
    \vdots & \vdots & \vdots  & \dotsc & \vdots  & \vdots & \vdots & \vdots & \vdots\\
    0 & 0 & 0  & \dotsc & e_{m-2} & 0 & 0 & \dotsc &l_{n-2}\\
    0 & 0 & 0  & \dotsc & e_{m-1}  & 0 & 0 & \dotsc & l_{n-1}\\
    0 & 0 & 0  &\dotsc & e_{m}  & 0 & 0 & \dotsc & l_{n}
    \end{bmatrix}.
\end{equation}
}
    \item [\rm (2)] If $m$ is even and $n$ is odd, then the matrix has the following form:
{\scriptsize
    \begin{equation}\label{matrix2of4}
    {\rm Syl}_{Q(a,-1,c)}(f,g)=\begin{bmatrix}
    e_{0} & 0 & 0  & \dotsc & 0  & l_{0} & al_{0} & \dotsc &0\\
    e_{1} & e_{0}+ae_{1} & 0  & \dotsc & 0  & l_{1} & -l_{0} & \dotsc & 0\\
    e_{2} & -e_{1} & e_{0}  & \dotsc & 0  & l_{2} & cl_{0}+l_{1}+al_{2} & \dotsc & 0\\
    e_{3} & ce_1+e_2+ae_{3} & e_{1}  & \dotsc & 0  & l_{3} & -l_{2} & \dotsc & 0\\
    e_{4} & -e_3 & e_{4} &\dotsc & 0  & l_{4} & cl_{2}+l_{3}+al_{4} & \dotsc & 0\\
    e_5 & ce_{3}+e_{4}+ae_{5} & e_{5}  & \dotsc & 0  & l_{5} & -l_{4} & \dotsc & 0\\
    \vdots & \vdots & \vdots & \vdots & \dotsc & \vdots  & \vdots &\vdots & \vdots  \\
    e_{m-2} & -e_{m-3} & e_{m-4}  & \dotsc & 0 & l_{n-3} & cl_{n-5}+l_{n-4}+al_{n-3} & \dotsc &0\\
    e_{m-1} & ce_{m-3}+e_{m-2}+ae_{m-1} & e_{m-3}  & \dotsc & 0  & l_{n-2} & -l_{n-3} & \dotsc & 0\\
    e_{m} & -e_{m-1} & e_{m-2}  & \dotsc & 0 & l_{n-1} & cl_{n-3}+l_{n-2}+al_{n-1} & \dotsc & 0\\
    0 &  ce_{m-1}+e_{m} & e_{m-1}  & \dotsc & e_{0}  & l_{n} & -l_{n-1} & \dotsc & l_{0}\\
    0 & 0 & e_{m}  & \dotsc & e_{1} & 0& cl_{n-1}+l_{n} & \dotsc & l_{n}\\
    \vdots & \vdots & \vdots & \vdots & \dotsc & \vdots  & \vdots & \vdots & \vdots\\
    0 & 0 & 0  & \dotsc & e_{m-2}  & 0 & 0 & \dotsc &l_{n-2}\\
    0 & 0 & 0  & \dotsc & e_{m-1} & 0 & 0 & \dotsc & l_{n-1}\\
    0 & 0 & 0  &\dotsc & e_{m}  & 0 & 0 & \dotsc & l_{n}
    \end{bmatrix}.
\end{equation}
}
\end{enumerate}
\end{landscape}

\begin{landscape}
\begin{enumerate}
    \item [\rm (3)] If $m$ is odd and $n$ is even, then the matrix has the following form:
{\tiny
    \begin{equation}\label{matrix3of4}
    {\rm Syl}_{Q(a,-1,c)}(f,g)=\begin{bmatrix}
    e_{0} & ae_0 & 0 & 0 & \dotsc & 0  & l_{0} & 0 & \dotsc & 0\\
    e_{1} & -e_{0} & 0 & 0 & \dotsc & 0  & l_{1} & l_{0} + al_{1} & \dotsc & 0\\
    e_{2} & ce_{0} + e_{1} + ae_{2} & e_{0} & ae_{0} & \dotsc & 0 & l_{2} & -l_{1} & \dotsc & 0\\
    e_{3} & -e_2 & e_{1} & -e_{0} & \dotsc & 0 & l_{3} & cl_{1} + l_{2} +al_3 & \dotsc & 0\\
    e_{4} & ce_2 + e_3 + ae_4 & e_{2}& ce_{0} + e_1 + ae_2 &\dotsc & 0 & l_{4} & -l_{3} & \dotsc & 0\\
    e_5 & - e_{4} & e_{3} & -e_{2} & \dotsc & 0 & l_{5} & cl_3 +l_{4} +al_5 & \dotsc & 0\\
    \vdots & \vdots & \vdots & \vdots & \dotsc & \vdots & \vdots &\vdots & \dotsc & \vdots  \\
    e_{m-2} & -e_{m-3} & e_{m-4} & -ce_{m-5} & \dotsc & ae_0  & l_{n-3} & cl_{m-5}+l_{m-4}+al_{m-3} & \dotsc & 0\\
    e_{m-1} & ce_{m-3} +e_{m-2}+ae_{m-1} & e_{m-3} & ce_{m-5} + e_{m-4} +ae_{m-3} & \dotsc & -e_0 & l_{n-2} & -l_{n-3} & \dotsc & l_0\\
    e_{m} & -e_{m-1} & e_{m-2} & -e_{m-3} & \dotsc & ce_0 + e_1 + ae_2 & l_{n-1} & cl_{n-3} + l_{n-2} + al_{n-1} & \dotsc & l_1\\
    0 &  ce_{m-1}+e_{m} & e_{m-1} & ce_{m-3} +e_{m-2} + ae_{m-1} & \dotsc & -e_2 & l_{n} & -l_{n-1} & \dotsc & l_{2}\\
    0 & 0 & e_{m} & -e_{m-1} & \dotsc & ce_{2} + e_3 + ae_4 & 0& cl_{n-1} +l_n & \dotsc & l_{3}\\
    \vdots & \vdots & \vdots & \vdots & \dotsc & \vdots & \vdots & \vdots & \dotsc & \vdots\\
    0 & 0 & 0 & 0 & \dotsc & ce_{m-3}+e_{m-2}+ae_{m-1} & 0 & 0 & \dotsc & l_{n-2}\\
    0 & 0 & 0 & 0 & \dotsc & -e_{m-1} & 0 & 0&  \dotsc &l_{n-1}\\
    0 & 0 & 0 & 0 &\dotsc &  ce_{m-1} + e_{m} & 0 & 0 & \dotsc & l_{n}
    \end{bmatrix}.
\end{equation}
}
    \item [\rm (4)] If $m$ is odd and $n$ is odd, then the matrix has the following form:
{\tiny
    \begin{equation}\label{matrix4of4}
   {\rm Syl}_{Q(a,-1,c)}(f,g)=\begin{bmatrix}
    e_{0} & ae_0 & 0 & 0 & \dotsc & 0  & l_{0} & al_0 & \dotsc & 0\\
    e_{1} & -e_{0} & 0 & 0 & \dotsc & 0  & l_{1} & -l_{0}  & \dotsc & 0\\
    e_{2} & ce_{0} + e_{1} + ae_{2} & e_{0} & ae_{0} & \dotsc & 0 & l_{2} & cl_0 + l_1 + al_2 & \dotsc & 0\\
    e_{3} & -e_2 & e_{1} & -e_{0} & \dotsc & 0 & l_{3} & -l_2 & \dotsc & 0\\
    e_{4} & ce_2 + e_3 + ae_4 & e_{2}& ce_{0} + e_1 + ae_2 &\dotsc & 0 & l_{4} & cl_2 +l_3 + al_4 & \dotsc & 0\\
    e_5 & - e_{4} & e_{3} & -e_{2} & \dotsc & 0 & l_{5} & -l_4 & \dotsc & 0\\
    \vdots & \vdots & \vdots & \vdots & \dotsc & \vdots & \vdots &\vdots & \dotsc & \vdots  \\
    e_{m-2} & -e_{m-3} & e_{m-4} & -ce_{m-5} & \dotsc & ae_0  & l_{n-3} & -l_{m-3} & \dotsc & 0\\
    e_{m-1} & ce_{m-3} +e_{m-2}+ae_{m-1} & e_{m-3} & ce_{m-5} + e_{m-4} +ae_{m-3} & \dotsc & -e_0 & l_{n-2} & cl_{n-3} + l_{n-2} + al_{n-1} & \dotsc & l_0\\
    e_{m} & -e_{m-1} & e_{m-2} & -e_{m-3} & \dotsc & ce_0 + e_1 + ae_2 & l_{n-1} & -l_{n-1} & \dotsc & l_1\\
    0 &  ce_{m-1}+e_{m} & e_{m-1} & ce_{m-3} +e_{m-2} + ae_{m-1} & \dotsc & -e_2 & l_{n} & cl_{n-1} + l_m & \dotsc & l_{2}\\
    0 & 0 & e_{m} & -e_{m-1} & \dotsc & ce_{2} + e_3 + ae_4 & 0& 0 & \dotsc & l_{3}\\
    \vdots & \vdots & \vdots & \vdots & \dotsc & \vdots & \vdots & \vdots & \dotsc & \vdots\\
    0 & 0 & 0 & 0 & \dotsc & ce_{m-3}+e_{m-2}+ae_{m-1} & 0 & 0 & \dotsc & l_{n-2}\\
    0 & 0 & 0 & 0 & \dotsc & -e_{m-1} & 0 & 0&  \dotsc &l_{n-1}\\
    0 & 0 & 0 & 0 &\cdots &  ce_{m-1} + e_{m} & 0 & 0 & \cdots & l_{n}
    \end{bmatrix}.
\end{equation}
}
\end{enumerate}
\end{landscape}

\begin{example}
Let $f(x,y) = x^{3}+y^{3}, g(x,y) = x^2y+xy^{2}\in Q(a,-1,c)$. Then $c(x,y) = c_{1}x^{2}+c_{2}xy+c_{3}y^{2}$ and $d(x,y) = d_{1}x^{2}+d_{2}xy+d_{3}y^{2}$. In this way,
\begin{align*}
    cf = &\ (c_{1}+ac_{2})x^{5}-c_{2}x^{4}y+(c_{3}-c)x^{3}y^{2}+c_{1}x^{2}y^{3}+c_{2}xy^{4}+c_{3}y^{5},\ \ {\rm and}\\
    dg = &\ d_{1}x^{4}y+(d_{1}+d_{2}+ad_{2})x^{3}y^{2}+(d_{3}-d_{2})x^{2}y^{3}+(d_{3}+d_{2}c)xy^{4}.
\end{align*}

The Sylvester resultant matrix of $f$ and $g$ is given by
\begin{equation*}
    {\rm Syl}_{Q(a,-1,c)}(f,g)=\begin{bmatrix}
    1 & a &0 & 0 & 0 & 0 \\
    0 & -1 &0 & 1 & 0 & 0 \\
    0 & -c &1 & 1 & 1 & a \\
    1 & 0 &0 & 1 & -1 & 0 \\
    0 & 1 &0 & 0 & c & 1 \\
    0 & 0 &1 & 0 & 0 & 0 \\
    \end{bmatrix}.
\end{equation*}
\end{example}

\subsubsection{Case $Q(a,0,0)$}
The Sylvester matrix ${\rm Syl}_{Q(a,0,0)}(f,g)$ is given by
{\normalsize{
\begin{align*}
\begin{bmatrix}
        e_{0} & ae_{0} & a^{2}e_{0}  & \hdots & a^{n-1}e_{0} & l_{0} & al_{0} & a^{2}l_{0}& \hdots & a^{m-1}l_{0}\\
        e_{1} & ae_{1} & a^{2}e_{1} &  & a^{n-1}e_{1}& l_{1} & al_{1} & a^{2}l_{1} & & a^{m-1}l_{1} \\
        e_{2} & ae_{2} & a^{2}e_{2}  & &a^{n-1}e_{2}&l_{2} & al_{2} & a^{2}l_{2}& & a^{m-1}l_{2}\\
        e_{3} & ae_{3} & a^{2}e_{3} & & a^{n-1}e_{3} & l_{3} & al_{3} & a^{2}l_{3}&\hdots & a^{m-1}l_{3}\\
        \vdots & \vdots & \vdots & \vdots & \cdots 
     & \vdots & \vdots & \vdots & \vdots & \vdots\\
        e_{m-2} & ae_{m-2} & a^{2}e_{m-2} & & a^{n-1}e_{m-2}& l_{n-3} & al_{n-3} & a^{2}l_{n-3} & & a^{m-1}l_{n-3}\\
        e_{m-1} & ae_{m-1} & a^{2}e_{m-1}& & a^{n-1}e_{m-1} & l_{n-2} & al_{n-2} & a^{2}l_{n-2} & &a^{m-1}l_{n-2}\\
        e_{m}& 0 & 0  & & 0 & l_{n-1} & al_{n-1} & a^{2}l_{n-1} &  &a^{m-1} l_{n-1}\\
        0 & e_{m} & 0& & 0 & l_{n} & 0 &0  & & 0\\
        0 & 0 & e_{m} &  & 0 & 0 & l_{n} & 0 & & 0\\
        0 & 0 & 0  &  & 0 & 0 & 0 & l_{n} & & 0\\
        \vdots  & \vdots & \vdots & \vdots & \cdots &\vdots& \vdots & \vdots & \vdots &  \\
        0 & 0 & 0  & \hdots & e_{m} & 0 & 0 & 0 & \hdots& l_{n}
    \end{bmatrix}.
\end{align*}
}}

\subsubsection{Case $Q(0,b,0)$}
The Sylvester matrix ${\rm Syl}_{Q(0,b,0)}(f,g)$ is given by
{\normalsize{
\begin{equation}\label{matrix_quantum_plane}
    \begin{bmatrix}
        e_{0} & 0& 0  & \cdots & l_{0} & \cdots &  0 & 0\\
        e_{1} & b^{m}e_{0} & 0  &  & l_{1} & & 0& 0 \\
        e_{2} & b^{m-1}e_{1} & b^{2m}e_{0}  & & l_{2}& &0 & 0\\
        e_{3} & b^{m-2}e_{2} & b^{2(m-1)}e_{1} & & l_{3} & &0 & 0\\
        \vdots & \vdots & \vdots  & \hdots& \vdots & \cdots &\vdots& \vdots\\
        e_{m-2} & b^{3}e_{m-3} & b^{8}e_{m-4} & & l_{n-1} & & 0 & 0\\
        e_{m-1} & b^{2}e_{m-2} & b^{15}e_{m-3}& & l_{n} & &0& 0\\
        e_{m}& be_{m-1} & b^{4}e_{m-2} & & 0 & & l_{0}b^{n(m-2)} &0\\
        0 & e_{m} & b^{2}e_{m-1}& & 0 & &l_{1}b^{(n-1)(m-2)}&  l_{0}b^{(n)(m-1)}\\
        0 & 0 & e_{m}& & 0 & & l_{2}b^{(n-2)(m-2)} & l_{1}b^{(n-1)(m-1)}\\
        0 & 0 & 0  & & 0 & & l_{3}b^{(n-3)(m-2)} & l_{2}b^{(n-2)(m-1)} \\
        \vdots    & \vdots & \vdots  & \cdots & \vdots & \cdots &\vdots & \vdots \\
        0 & 0 & 0 & & 0 &  & l_{n-1}b^{2(m-2)} & l_{n-2}b^{2(m-1)}\\
        0 & 0 & 0 &  & 0 & & l_{n}b^{(m-2)} & l_{n-1}b^{(m-1)}\\
        0 & 0 & 0 & \cdots & 0 & \cdots & 0 & l_{n} 
    \end{bmatrix}.
\end{equation}
}}

\subsection{Resultants and right common factors}\label{commonfactors}

It is well-known that the concept of {\em resultant} can be introduced by asking when two polynomials in the commutative polynomial ring $\Bbbk[x]$ have a common factor. Two important applications of resultant theory are elimination theory, and the proofs of {\em Extension theorem} and {\em Bezout's theorem} (see Cox et al. \cite[Chapter 3]{Coxetal2015} for more details). In the commutative and some noncommutative cases, it has been proven that the existence of common factors is equivalent to the existence of polynomials which satisfy expression (\ref{defining_equation}) \cite{ArtamanovCohn1999, Eric2008, Morozov2010, RuedaSendra, ZhangYuanGao}. This implies that the Sylvester matrix and the notion of resultant encodes the existence of common factors (right factors in the noncommutative case) of a pair of polynomials and homogeneous polynomials in commutative multivariate algebras. Next, we explore whether the resultant is a complete criterion to decide the existence of common right factor for a pair of homogeneous polynomials in the quadratic algebra defined by the relation (\ref{GolovashkinMaksimov2005(1)}). 

\begin{example}
\begin{enumerate}
    \item [\rm (i)] Let $f = x^{2}+(1-b)xy-y^{2}=(x-y)(x+y)$ and $g = x^{2}+(1+b)xy+y^{2} = (x+y)^{2}$ be polynomials in $Q(0,b,0)$. Then $x+y$ is a common right factor of $f$ and $g$, and according to expression (\ref{matrix_quantum_plane}), its Sylvester matrix is given by
\begin{equation*}
        {\rm Syl}_{Q(0,b,0)}(f,g)=\begin{bmatrix}
1 & 0 & 1 & 0 \\
1-b & b & 1+b &  b \\
-1 & (1-b)b & 1 & (1+b)b \\
0 & -b & 0 & b
\end{bmatrix},
\end{equation*}
    
which implies that ${\rm Res}_{Q(0,b,0)}(f,g)=-4b^{3}+4b^{2}=-4(b^{2})(b-1)$. In this way, $f$ and $g$ have a common right factor but ${\rm Res}_{Q(0,b,0)}(f,g)=0$. 

\item [\rm (ii)] Let us take $f(x,y)=e_{0}x^{2}+e_{2}y^{2}$, $g(x,y)=l_{0}x^{2}+l_{2}y^{2}$ be elements in $Q(a,0,0)$ with $e_{0}, e_{2}, l_{0}$ and $l_{2}$ being non-zero elements such that $e_{0}l_{2}\neq e_{0}l_{2}$. Then
\begin{equation*}
        {\rm Syl}_{Q(a,0,0)}(f,g)=\begin{bmatrix}
            e_{0} & ae_{0} & l_{0} & al_{0}\\
            0 & 0 & 0 & 0 \\
            e_{2} & 0 & l_{2} & 0 \\
            0 & ae_{2} & 0 & al_{2}
        \end{bmatrix},
\end{equation*}
    
whence ${\rm Res}_{Q(a,0,0)}=0$. However, these polynomials do not have a common right factor. If there exists a polynomial $p(x,y)=b_{0}x+b_{1}y$ such that 
\begin{align*}
    e_{0}x^{2}+e_{2}y^{2}&=(a_{0}x+a_{1}y)(b_{0}x+b_{1}y)\\
     l_{0}x^{2}+l_{2}y^{2}&=(c_{0}x+c_{1}y)(b_{0}x+b_{1}y),
\end{align*}

we can see that the system has a solution if and only if $e_{0}l_{2}=e_{0}l_{2}$, a contradiction.

\item [\rm (iii)] Consider the polynomials $f(x,y) =(1+a)x^{2}+(1+c)y^{2}=(x+y)^{2}$ and $g(x,y) = (1-a)x^{2}+2xy+(1-c)y^{2}=(x-y)(x+y)$ in the quadratic algebra $Q(a,-1,c)$. Then
\begin{equation*}
        {\rm Syl}_{Q(a,-1,c)}(f,g)=\begin{bmatrix}
            1+a & 0 & 1-a & 0\\
            0 & 1+a & 2 & 1+a \\
            1+c & 0 & 1-c & -2 \\
            0 & 1+c & 0 & 1+c
        \end{bmatrix}.
\end{equation*}

This fact implies that ${\rm Res}_{Q(a,-1,c)}(f,g)=-4(a+ac+c+1) \neq 0$. Again, the fact that $f$ and $g$ have common right factor does not imply that ${\rm Res}_{Q(a,-1,c)}(f,g)=0$. On the other hand, if we take the following polynomials $f(x,y) = x^2$ and $g(x,y) = xy$, it can be seen that
    \begin{equation*}
        {\rm Syl}_{Q(a,-1,c)}(f,g)=\begin{bmatrix}
            1 & 0 & 0 & 0\\
            0 & 1 & 1 & a \\
            0 & 0 & 0 & -1 \\
            0 & 0 & 0 & c
        \end{bmatrix},
    \end{equation*}

which implies that ${\rm Res}_{Q(a,-1,c)}=0$. However $f$ and $g$ cannot have a common right factor: all possible factorizations of $x^{2}$ are given by $x^{2} = (kx)(k^{-1}x)$ and $x^{2} = \frac{k}{k+aw}(x-cy)(x-cy)$, while the unique possible factorizations of $xy$ are $xy=(kx)(k^{-1}y)$, so the assertion follows.
\end{enumerate}
\end{example}

The following theorem shows that the behavior of this initial version of resultant via Sylvester matrix differs significantly from the classical one, as it implies that the resultant becomes zero in cases where the polynomials have the same degree.

\begin{theorem}
    Let $f(x,y), g(x,y)$ be two polynomials in $Q(a,0,0)$ having the same total degree $n$. Then ${\rm Res}_{Q(a,0,0)}(f,g)=0$.
\end{theorem}
\begin{proof}
    Let $f,g$ be two polynomials in $Q(a,0,0)$ such that $f(x,y) = \sum\limits_{i=0}^{n}e_{i}x^{n-i}y^{i}$ and $g(x,y) = \sum\limits_{j=0}^{n}l_{j}x^{n-j}y^{j}$. The Sylvester matrix of $f$ and $g$, ${\rm Syl}_{Q(a,0,0)}(f,g)$, is a square matrix $2n \times 2n$ of the form
{\normalsize{
\begin{align*}
    \begin{bmatrix}
        e_{0} & ae_{0} & a^{2}e_{0}  & \hdots & a^{n-1}e_{0} & l_{0} & al_{0} & a^{2}l_{0}& \hdots & a^{n-1}l_{0}\\
        e_{1} & ae_{1} & a^{2}e_{1} &  & a^{n-1}e_{1}& l_{1} & al_{1} & a^{2}l_{1} & & a^{n-1}l_{1} \\
        e_{2} & ae_{2} & a^{2}e_{2}  & &a^{n-1}e_{2}&l_{2} & al_{2} & a^{2}l_{2}& & a^{n-1}l_{2}\\
        e_{3} & ae_{3} & a^{2}e_{3} & & a^{n-1}e_{3} & l_{3} & al_{3} & a^{2}l_{3}&\hdots & a^{n-1}l_{3}\\
        \vdots & \vdots & \vdots & \vdots & \cdots 
     & \vdots & \vdots & \vdots & \vdots & \vdots\\
        e_{n-2} & ae_{n-2} & a^{2}e_{n-2} & & a^{n-1}e_{n-2}& l_{n-2} & al_{n-2} & a^{2}l_{n-2} & & a^{n-1}l_{n-2}\\
        e_{n-1} & ae_{n-1} & a^{2}e_{n-1}& & a^{n-1}e_{n-1} & l_{n-1} & al_{n-1} & a^{2}l_{n-1} & &a^{n-1}l_{n-1}\\
        e_{n}& 0 & 0  & & 0 & l_{n} & 0 & 0 &  &0 \\
        0 & e_{n} & 0& & 0 & 0 & l_{n} &0  & & 0\\
        0 & 0 & e_{n} &  & 0 & 0 & 0 & l_{n} & & 0\\
        \vdots  & \vdots & \vdots & \vdots & \cdots &\vdots& \vdots & \vdots & \vdots &  \\
        0 & 0 & 0  & \hdots & e_{n} & 0 & 0 & 0 & \hdots& l_{n}
    \end{bmatrix}.
\end{align*}
}}

It is easy to see that some columns can be reduced to zero. For example, the second column has the coefficients of the polynomial $f(x,y)$ multiplied by $a$ (except for the last coefficient), as it also happens with the column $n+2$ where we have the coefficients of $g(x,y)$ multiplied by $a$. Thus, the matrix can be reduced by operations between columns, in this case, multiplying the first column and column $n+1$ by $a$ and subtracting with columns $2$ and $n+2$, respectively. With this operation, we obtain the following equivalent matrix:
\begin{align*}
\begin{bmatrix}
        e_{0} & 0 & a^{2}e_{0}  & \hdots & a^{n-1}e_{0} & l_{0} & 0 & a^{2}l_{0}& \hdots & a^{n-1}l_{0}\\
        e_{1} & 0 & a^{2}e_{1} &  & a^{n-1}e_{1}& l_{1} & 0 & a^{2}l_{1} & & a^{n-1}l_{1} \\
        e_{2} & 0 & a^{2}e_{2}  & &a^{n-1}e_{2}&l_{2} & 0 & a^{2}l_{2}& & a^{n-1}l_{2}\\
        e_{3} & 0 & a^{2}e_{3} & & a^{n-1}e_{3} & l_{3} & 0 & a^{2}l_{3}&\hdots & a^{n-1}l_{3}\\
        \vdots & \vdots & \vdots & \vdots & \cdots 
     & \vdots & \vdots & \vdots & \vdots & \vdots\\
        e_{n-2} & 0 & a^{2}e_{n-2} & & a^{n-1}e_{n-2}& l_{n-2} & 0 & a^{2}l_{n-2} & & a^{n-1}l_{n-2}\\
        e_{n-1} & 0 & a^{2}e_{n-1}& & a^{n-1}e_{n-1} & l_{n-1} & 0 & a^{2}l_{n-1} & &a^{n-1}l_{n-1}\\
        e_{n}& 0 & 0  & & 0 & l_{n} & 0 & 0 &  &0 \\
        0 & e_{n} & 0& & 0 & 0 & l_{n} &0  & & 0\\
        0 & 0 & e_{n} &  & 0 & 0 & 0 & l_{n} & & 0\\
        \vdots  & \vdots & \vdots & \vdots & \cdots &\vdots& \vdots & \vdots & \vdots &  \\
        0 & 0 & 0  & \hdots & e_{n} & 0 & 0 & 0 & \hdots& l_{n}
    \end{bmatrix}.
\end{align*}

Now, subtracting the columns $2$ and $n+2$ multiplied by $e_n/l_n$, we can reduce the second column to zero. Thus, after some operations, the Sylvester matrix is equivalent to a matrix that has a column with zero entries, so that its determinant is zero, i.e. ${\rm Res}_{Q(a,0,0)}(f,g)=0$. 
\end{proof}
\begin{remark}
    We compare the form of the Sylvester matrix for the same values of $f$ and $g$ seeing how different can be according to the values of the parameters $a, b$ and $c$. Let us take:
\begin{align*}  f(x,y)&=e_{0}x^{3}+e_{1}x^{2}y+e_{2}xy^{2}+e_{3}y^{3}\\
g(x,y)&=l_{0}x^{2}+l_{1}xy+l_{2}y^{2}.
\end{align*}

In this way, the matrix has the following form in every possibility:
\begin{align*}
    {\rm Syl}_{Q(a,0,0)}(f,g) = &\ \begin{bmatrix}
e_{0} & ae_{0} & l_{0} & al_{0} & a^{2}l_{0}\\
e_{1} & ae_{1} & l_{1} & al_{1} & a^{2}l_{1} \\
e_{2} & ae_{2} & l_{2} & 0 & a^{2}l_{2} \\
e_{3} & 0 & 0 & l_{2} &0\\
0 & e_{3} & 0 & 0 & l_{2}
\end{bmatrix}, \\
{\rm Syl}_{Q(0,b,0)}(f,g) = &\ \begin{bmatrix}
e_{0} & 0 & l_{0} & 0 & 0\\
e_{1} & b^{3}e_{0} & l_{1} & b^{2}l_{0} & 0 \\
e_{2} & b^{2}e_{1} & l_{2} & bl_{1} & b^{4}l_{0} \\
e_{3} & be_{2} & 0 & l_{2} & b^{2}l_{1}\\
0 & e_{3} & 0 & 0 & l_{2}
\end{bmatrix},\\
{\rm Syl}_{Q(a,-1,c)}(f,g) = &\ \begin{bmatrix}
e_{0} & e_{0}a & l_{0} & 0 & 0\\
e_{1} & -e_{0} & l_{1} & l_{0}+al_{1} & 0 \\
e_{2} & e_{0}c+e_{1}+e_{2}a & l_{2} & -l_{1} & l_{0}\\
e_{3} & -e_{2} & 0 & b+l_{2} & l_{1}\\
0 & e_{2}+e_{3} & 0 & 0 & l_{2}
\end{bmatrix}.
\end{align*}
\end{remark}

\section{Burchnall-Chaundy theory for quadratic algebras}\label{DeterminantPolynomials}

Following \cite[Section 3]{Larsson2014}, we generalize the notion of resultant to the context of quadratic algebras. The key fact of our treatment is that $Q(a,b,c)$ is an $\Bbbk$-algebra finitely generated over $\Bbbk[x]$ because the set $\{x^{i}y^{j}\mid i,j \geq 0\}$ is a PBW basis. The following definition is analog to the concept of determinant polynomial found in \cite[section 3]{Larsson2014}, which is a matrix concept that is, in fact, independent of the noncommutative structure of algebra. 

\begin{definition}[{\cite[Definition 1.3.1]{LiPhD1996}; \cite[p. 241]{Mishra1993}}]
    Let $M\in M_{r\times c}(\Bbbk[x])$. Then we define the {\em determinant polynomial of} $M$, denoted by $|M|$, in the following way:
\begin{equation*}
    |M|=\sum_{i=0}^{c-r} {\rm det}(M_{i})y^{i},
\end{equation*}

where $M_{i}$ is the square matrix that satisfies the following properties:
\begin{enumerate}
    \item [\rm (i)] The first $r-1$ columns of $M_i$ are the same that the first $r-1$ columns of $M$.
    \item [\rm (ii)] The last column of $M_i$ is the $(c-i)$th column of $M$.
\end{enumerate}
\end{definition}

An important remark is that the following proposition remains valid for $Q(a,b,c)$. This is due to the fact that the calculations involved in the determinant will be done under certain assumptions about the order in which the multiplications are done.

\begin{proposition}[{\cite[Proposition B.2.1]{RichterPhD2014}}]\label{Richter2014PropositionB.2.1}
Let $M\in M_{r\times c}(\Bbbk[x])$ with determinant polynomial $|M|$. For $i = 1,\ldots,r$, if $H_{i}$ is the polynomial given by
\begin{equation*}
    H_i = m_{i1}y^{c-1}+\cdots+m_{ir}y^{c-r}+\cdots+m_{ic},
\end{equation*}
then
\begin{equation*}
    |M|={\rm det}\begin{bmatrix}
    m_{1,d} & \dotsc & m_{1,d-r+1} & H_{1}\\
    m_{2,d} & \dotsc & m_{2,d-r+1} & H_{2}\\
    \vdots & \dotsc & \vdots & \vdots\\
    m_{r,d} & \dotsc & m_{r,d-r+1} & H_{r}
    \end{bmatrix}.
\end{equation*}    
\end{proposition}

In the expansion of this determinant, the elements not involving $y$ will be always multiplied from the left. This will be denoted as $\texttt{mult}(ay^{i},x):=xay^{i}$. In this way, there is no any problem with the expansion of this determinant. 

Having in mind the definition presented in \cite[Definition 3.1]{Larsson2014}, we introduce the concept of determinant polynomial associated to a sequence of polynomials in $Q(a,b,c)$.

\begin{definition}[{\cite[Section B.2.1]{Richter2014}}]
Let $G:=(g_1, g_2,\ldots, g_r)$ be a sequence of polynomials in $Q(a,b,c)$, and let $d$ be the maximum degree of the polynomials. Assume $d\geq r$. We define the {\em matrix of size} $r\times (d+1)$,  denoted by $M(G)$, whose entry in the $i$th row and $j$th column is the coefficient of the monomial $y^{d+1-j}$ in $g_{j}$. The determinant polynomial of $G$ is $|M(G)|$ and it is denoted by $|G|$.
\end{definition}

The next proposition describes an easy way to calculate the determinant polynomial of a sequence of polynomials in $Q(a,b,c)$. 

\begin{proposition}[{\cite[Proposition B.2.2]{RichterPhD2014}}]
Let $G:=(g_{1}, g_{2},\ldots,g_{r})$ be a sequence of elements of $Q(a,b,c)$ of maximum degree $d$ with respect to $y$. Then
    \begin{equation*}
    |G|={\rm det}\begin{bmatrix}
    a_{1,d} & \dotsc & a_{1,d-r+1} & g_{1}\\
    a_{2,d} & \dotsc & a_{2,d-r+1} & g_{2}\\
    \vdots & \dotsc & \vdots & \vdots\\
    a_{r,d} & \dotsc & a_{r,d-r+1} & g_{r}
    \end{bmatrix}.
\end{equation*}
where $a_{i,j}$ is the coefficient of $y^{j}$ in $g_{j}(x,y)$ written in its normal form.
\end{proposition}
\begin{proof}
This proposition is formulated within the same assumptions as in \cite[p. 133]{Li1998}. The assertion follows from Proposition \ref{Richter2014PropositionB.2.1} since the matrix entries correspond to the coefficients of the given polynomials.
\end{proof}

The following definition is motivated by Richter and Silvestrov \cite[Section 2.2]{RichterSilvestrov2012}.

\begin{definition}\label{resultant}
Let $f$ and $g$ be two elements of $Q(a,b,c)$ of degree $n$ and $m$, respectively. The {\em resultant of} $f$ {\em and} $g$, denoted  $\texttt{Res}(f,g)$, is the determinant polynomial of the sequence
\begin{equation*}
    f, yf, y^{2}f,\ldots, y^{m-1}f,g,yg,y^{2}g,\ldots,y^{n-1}g.
\end{equation*}
\end{definition}

An important consequence that can be visualized is the following property.

\begin{proposition}\label{expansion}
    For $f$ and $g$ polynomials in $Q(a,b,c)$, 
\begin{equation*}
    \texttt{Res}(f,g)=F_{1}(x)f+F_{2}(x)g.
\end{equation*}
\end{proposition}
\begin{proof}
    Notice that according to Definition \ref{resultant}, we have
\begin{align*}
    \texttt{Res}(f,g)={\rm det}\begin{bmatrix}
    a_{1,d} & \dotsc & a_{1,d-r+1} & f\\
    a_{2,d} & \dotsc & a_{2,d-r+1} & yf\\
    \vdots & \dotsc & \vdots & \vdots\\
    a_{r,d} & \dotsc & a_{r,d-r+1} & y^{n-1}f\\
     b_{1,d} & \dotsc & b_{1,d-r+1} & g\\
    b_{2,d} & \dotsc & b_{2,d-r+1} & yg\\
    \vdots & \dotsc & \vdots & \vdots\\
    b_{r,d} & \dotsc & b_{r,d-r+1} & y^{m-1}g
\end{bmatrix}.
\end{align*}
    
The expansion of this determinant along the last column shows the result has the expression that we are describing.
\end{proof}

The most important application for our interest is the following result, which establish a version of Burchnall-Chaundy theory for the cases of $Q(a,b,c)$ that we are considering.

\begin{theorem}\label{BC_corollary}
Let $f$ and $g$ be two polynomials such that $fg=gf$. Then there exists a polynomial $F(s,t)$ such that $F(f,g)=0$.
\begin{proof}
    Let $s, t$ be two new indeterminates such that they commute with all the elements in $Q(a,b,c)$. Let us establish that the $s$ and $t$ have degree zero, then according to Proposition \ref{expansion}, the following equality holds for some polynomials:
\begin{equation*}
        F(s,t)=\texttt{Res}(f-s,g-t)=F_{1}(x)(f-s)+F_{2}(x)(g-t)
\end{equation*}
    
Since $fg = gf$, it is not ambiguous the evaluation of $f$ and $g$ in $F(s,t)$, and then, the fact that $F(f,g)=0$.
\end{proof}
\end{theorem}

As we have mentioned previously, our approach is fully inspired by exploring the application of the ideas exposed by Li \cite{Li1998} in our context. It is worth noting that Li develops a theory not only of resultants but more generally about subresultants. In fact, the concept of subresultant generalizes what we have described regarding the notion of a resultant, as it can be seen in the following definition.

\begin{definition}
    Let $f(x,y)$ and $g(x,y)$ be two polynomials in $Q(a,b,c)$ of degree $n$ and $m$ respectively, with respect to the variable $y$. Suppose $n\geq m$, the $l$th \textit{subresultant} of $f$ and $g$ is
    \begin{align*}
        \texttt{sRes}_{l}(f,g) := |y^{m-l-1}f, \ldots, yf,f,y^{n-l-1}g,\ldots,yg,g| 
    \end{align*}
\end{definition}

The $0$th subresultant $\texttt{sRes}_{0}(f, g)$ is in fact $\texttt{Res}(f, g)$.

\section{Examples}\label{ExamplesBC}

We are going to present some illustrative examples of the results obtained in the previous sections.

\begin{example}
    Let us take the polynomials $ f = x^{2}y^2, g = (x^{3}+x)y^2 \in Q(a,0,0)$. It can be seen that $fg=gf$, whence
\begin{align*}
    F(s,t) = &\ \texttt{Res}(f-s,g-t) \\
    = &\ {\rm det}\begin{bmatrix}
   -s & 0 & x^{2} & x^2y^2-s\\
    0& -s  & ax^3 & y(x^2y^2-s)\\
   -t & 0  & x^{3}+x& (x^{3}+x)y^2-t\\
   0 & -t  & a(x^4+x^2) & y\left((x^{3}+x)y^2-t\right)    \end{bmatrix}\\
   = &\ {\rm det}\begin{bmatrix}
   -s & 0 & x^{2} & (x^2y^2-s)\\
    0& -s  & ax^3 & ax^{3}y^{2}-sy\\
   -t & 0  & x^{3}+x & (x^{3}+x)y^2-t\\
   0 & -t  & a(x^4+x^2) & a(x^{4}+x^{2})y^{2}-ty  \end{bmatrix} \\
   = &\ 0.
\end{align*}
\end{example}

\begin{example}\label{coef_example_1}
For the polynomials $f = y^{2}, g = (x^2+1)y^2+1$ in $Q(a,0,0)$, we have $fg=gf=$, and
\begin{align*}
    F(s,t) = &\ \texttt{Res}(f-s,g-t) \\ = &\ {\rm det}\begin{bmatrix}
    -s & 0 & 1 & y^2-s\\
    0 & -s & 0 & y^3-ys\\
    1-t & 0 & x+1 & (x+1)y^2+(1-t) \\
    0 & (1-t) & ax^{2} & y^3+ax^2y^2+(1-t)y
   \end{bmatrix}\\
   = &\ s^2y^3+2sy^3+t^2y^3-2sty^3-2ty^3+s^2xy^3+sxy^3-stxy^3+y^3\\
   = &\ (s^2+2s+t^2-2st-2t+s^2x+sx-stx+1)y^3.
\end{align*}
\end{example}

The previous examples give us some ideas about the behavior of the resultant in $Q(a,0,0)$. Next, we give sufficient conditions on $f(x,y)$ and $g(x,y)$ for $\texttt{Res}(f,g) = 0$. First, we mention a case where the resultant is non-zero.

\begin{proposition}\label{linear-case}
For any pair of commuting elements of degree 1 in $Q(a, 0, 0)$, $F(s,t)=\texttt{Res}(f-s,g-t)\neq 0$.
\end{proposition}
\begin{proof}
Let $f=p_{0}+p_{1}y$ and $g=q_{0}+q_{1}y$ be two commuting elements in $Q(a,0,0)$. Then
\begin{align*}
     F(s,t) = \texttt{Res}(f-s,g-t) & ={\rm det}\begin{bmatrix}
   p_{0}-s & p_{0}-s+p_{1}y \\
    q_{0}-t&  q_{0}-t+q_{1}y \\
     \end{bmatrix}\\
     & = (p_{0}-s)(q_{0}-t+q_{1}y) - (q_{0}-t)(p_{0}-s+p_{1}y) \\
     & = (p_{1}t-q_{1}s+p_{0}q_{1}-q_{0}p_{1})y.
\end{align*}

Since $y$ is a non-zero divisor, if $F(s,t) := p_{1}t-q_{1}s+p_{0}q_{1}-q_{0}p_{1}$, then
\begin{align*}
    F(f,g)&=p_{1}(q_{0}+q_{1}y)-q_{1}(p_{0}+p_{1}y)+p_{0}q_{1}-q_{0}p_{1}\\
    &= p_{1}q_{0}+p_{1}q_{1}y-q_{1}p_{0}-q_{1}p_{1}y+p_{0}q_{1}-q_{0}p_{1} = 0.
\end{align*}
\end{proof}

\begin{remark}
In the proof of Proposition \ref{linear-case} we do not use the noncommutative structure of the algebra; it only uses that polynomials in the indeterminate $x$ are commutative which holds in every algebra of our interest in the paper.
\end{remark}

\begin{lemma}\label{zero_principal}
    Let $f(x,y)=\sum_{i=0}^{n} p_{i}(x)y^{i}$ be a polynomial of degree $m$ with respect to the indeterminate $y$ in $Q(a,0,0)$. If $x$ is a factor of the polynomial $p_{n}(x)$, then $y^{m}f(x,y)$ is a polynomial of degree less than $m+n$ with respect to $y$.
    \begin{proof}
        Since $x$ is a factor of $p_{n}(x)$, there exists a nonzero polynomial $r(x)$ such that $p_{n}(x)=xr(x)$. Then
        \begin{align*}
            y^{m}f(x,y)&= y^{m}\left( \sum_{i=0}^{n} p_{i}(x)y^{i}\right) = \sum_{i=0}^{n} y^{m}p_{i}(x)y^{i} = y^{m}p_{n}(x)y^{n}+\sum_{i=0}^{n-1} y^{m}p_{i}(x)y^{i}\\ 
            &= y^{m}xr(x)y^{n} + \sum_{i=0}^{n-1} y^{m}p_{i}(x)y^{i} = a^{m}x^{m+1}r(x)y^{n} + \sum_{i=0}^{n-1} y^{m}p_{i}(x)y^{i},
        \end{align*}
        since the polynomial $\sum_{i=0}^{n-1} y^{m}p_{i}(x)y^{i}$ has degree at most $m+(n-1)$ with respect to $y$. It follows that $y^{m}f(x,y)$ has degree at most $m+(n-1)$.
    \end{proof}
\end{lemma}

\begin{theorem}
    Let $f(x,y) = \sum_{i=0}^{n} p_{i}(x)y^{i}$ and $g(x,y) = \sum_{j=0}^{m}q_{j}(x)y^{j}$ be two polynomials of $Q(a,0,0)$ with $m>1$ presented in their normal form. If $x$ is a factor of $p_{n}(x)$ and $q_{m}(x)$, then {{\rm \texttt{Res}}}$(f,g) = 0$. 
\end{theorem}
\begin{proof}
Let $f(x,y) = \sum_{i=0}^{n} p_{i}(x)y^{i}$ and $\sum_{j=0}^{m}q_{j}(x)y^{j}$ be two polynomials presented in their normal form. Since we assume the degree of $g$ greater than 1, this case is not the mentioned in Proposition \ref{linear-case}. Now, from the definition, we can compute $\texttt{Res}(f,g)$ as the determinant polynomial of the matrix 
{\tiny
\begin{equation*}
    \begin{bmatrix}
         p_0 & p_{1}  & \dotsc & p_n & 0 & 0 & \dotsc & 0\\   
         \pi_{0}(yf) & \pi_{1}(yf)  & \dotsc  &  \pi_{n}(yf) &  \pi_{n+1}(yf) & 0 & \dotsc  & 0\\
         \pi_{0}(y^2f) & \pi_{1}(y^2f)  & \dotsc  &  \pi_{n}(y^2f) &  \pi_{n+1}(y^2f) & \pi_{n+2}(y^2f) & \dotsc  & 0\\
            \vdots     &         \vdots    &      \vdots    &    \vdots   &   \vdots  &   \vdots  &        \vdots   &  \vdots  \\
         \pi_{0}(y^{m-1}f) & \pi_{1}(y^{m-1}f)  & \dotsc  &  \pi_{n}(y^{m-1}f) &  \pi_{n+1}(y^{m-1}f) & \pi_{n+2}(y^{m-1}f) & \dotsc  & \pi_{n+m-1}(y^{m-1}f)\\   
         q_0 & q_{1}  & \dotsc  & q_n & q_{n+1} & 0 & \dotsc  & 0\\ 
         \pi_{0}(yg) & \pi_{1}(yg)  & \dotsc  &  \pi_{n}(yg) &  \pi_{n+1}(yg) & 0 & \dotsc  & 0\\
         \pi_{0}(y^2g) & \pi_{1}(y^2g)  & \dotsc  &  \pi_{n}(y^2g) &  \pi_{n+1}(y^2g) & \pi_{n+2}(y^2g) & \dotsc  & 0\\
         \vdots     &     \vdots       &      \vdots    &    \vdots   &   \vdots  &   \vdots  &        \vdots   &  \vdots  \\
         \pi_{0}(y^{n-1}g) & \pi_{1}(y^{n-1}g)  & \dotsc  &  \pi_{n}(y^{n-1}g) &  \pi_{n+1}(y^{n-1}g) & \pi_{n+2}(y^{n-1}g) & \dotsc  & \pi_{n+m-1}(y^{n-1}g)\\
         \end{bmatrix}.
\end{equation*}
}

By using Lemma \ref{zero_principal}, we conclude that the entries along the last column are zero. This fact implies that the resultant is zero.        
\end{proof}

\begin{example}
Consider the polynomials $f = x^{2}y^2+xy, g = xy \in Q(0,b,0)$. Then
\begin{align*}
    fg &= \left(x^{2}y^{2}+xy\right)\left(xy\right) = b^2x^3y^3+bx^2y^2,\quad {\rm and} \\
    gf &= \left(xy\right)\left(x^{2}y^{2}+xy\right) = b^2x^3y^3+bx^2y^2.
\end{align*}

Hence,
\begin{align*}
    F(s,t) = &\ \texttt{Res}(f-s,g-t)={\rm det}\begin{bmatrix}
    -s & x & x^2y^{2}+xy-s\\
    -t & x & xy-t \\
    0 & -t & bxy^2-ty
     \end{bmatrix} = \left(-bs+bt+t^2\right)x^2y^2.
\end{align*}

Since $x^2y^{2} \neq 0$, it is easy to see that $F(s,t) = t^{2}-bs+bt$, and so
\[
    F(f,g) = (xy)^{2} - b(x^{2}y^2+xy) + b(xy) = bx^{2}y^{2} -bx^{2}y^{2}-bxy+bxy = 0.
\]
\end{example}

\begin{example}\label{coef_example_2}
Let $f = y^2, g = x^2y^2+x^2y\in Q(a,-1,c)$. It follows that
\begin{align*}
    fg &= \left(y^2\right)\left(x^2y^2+x^2y\right) = x^{2}y^{4} + x^{2}y^{3}, \quad {\rm and}  \\ 
    gf &= \left(x^2y^2+x^2y\right)\left(y^2\right) = x^{2}y^{4} + x^{2}y^{3},
\end{align*}

whence
\begin{align*}
    F(s,t) &\ = \texttt{Res}(f-s,g-t) = {\rm det}\begin{bmatrix}
-s & 0 & 0 & y^{2}-s\\
0 & -s & 0 & y^{3}-sy\\
-t & x^2 & x^{2} & x^2y^2+x^2y-t \\
0 & -t & x^{2} & x^{2}y^{3} + x^{2}y^{2} -ty
     \end{bmatrix}\\
     &\ = s^2x^4y^3 - sx^4y^3 - stx^2y^3 + stx^2y^2.
\end{align*}

After some computations, $F(s,t) = x^{2}ys - (y-1)t - x^{2}y$. It can be seen that
\begin{align*}
    F(f,g) = x^{2}y^{3} - (y-1)(x^{2}y^{2}+x^{2}y)-x^{2}y = 0.
\end{align*}

This example allows us to illustrate the $l$th subresultant of $f$ and $g$ with $l=1$:
\[
F_1(s,t) = \texttt{sRes}_{1}(f-s,g-t) = {\rm det}\begin{bmatrix}
-s & y^2-s \\
-t & x^2y^2+x^2y-t \end{bmatrix} = ty-x^{2}(y+1)s.
\]

By replacing $f$ and $g$, we get $F_{1}(f,g) = (x^{2}y^{2})y-x^{2}(y+1)y^2 = 0$.
\end{example}

Examples \ref{coef_example_1} and \ref{coef_example_2} show an interesting behavior with respect to the coefficients of the curve. The same behavior can be seen in \cite[Example 5.3] {Larsson2014} in the sense that the curve is not defined strictly on the coefficient field $\Bbbk$. A natural question would  be to determine in which cases the curve must have coefficients in $\Bbbk$.

\section{Future work}

Rosengren \cite{Rosengren2000} proved a {\em binomial formula} for the quadratic algebras studied in this paper, so we consider interesting to continue investigating a $\mathcal{BC}$ theory for these algebras that is as general as possible. Also, the research on $\mathcal{BC}$ theory for noncommutative algebras defined by two indeterminates can be continued for {\em PBW deformations of Artin-Schelter regular algebras} investigated by Gaddis in his PhD thesis \cite{GaddisPhD2013, Gaddis2016} and families of more general noncommutative algebras such as those considered by Lezama et al. \cite{LezamaLatorre2017, LezamaVenegas2020} (see also \cite{NinoReyes2023, ReyesSarmiento2022}). 

On the other hand, recently Previato et al. \cite{Previatoetal2023} defined the {\em matrix differential resultant} and use it to compute the Burchnall-Chaundy curve $\mathcal{BC}$ of a pair of commuting matrix ordinary differential operators (MODOs). They proved that this resultant provides the necessary and sufficient condition for the spectral problem to have a solution, and then by considering the Picard-Vessiot theory (c.f. \cite{Larsson2014}), they described explicitly isomorphisms between commutative rings of (MODOs) and a finite product of rings of irreducible algebraic curves. It is interesting to investigate the matrix differential resultant and its consequences for the quadratic algebras considered in this paper.

\end{document}